\documentclass[a4paper,10pt,reqno]{amsart}
\usepackage[utf8]{inputenc}
\usepackage{mathtools}
\usepackage{amssymb}
\usepackage{amsfonts}
\usepackage{amsthm}
\usepackage{bbm}
\usepackage{fullpage}
\usepackage{color}
\usepackage{hyperref}
\usepackage{stmaryrd}
\usepackage{enumitem}
\usepackage{amsaddr}

\newcommand{\N}{\mathbb N}
\newcommand{\Z}{\mathbb Z}
\newcommand{\R}{\mathbb R}

\newcommand{\eps}{\varepsilon}
\renewcommand{\P}{\mathbb P}
\newcommand{\E}{\mathbb E}

\newcommand{\F}{\mathcal F}
\newcommand{\ind}{\mathbbm 1}

\newtheorem{theorem}{Theorem}[section]

\theoremstyle{remark}
\newtheorem{example}[theorem]{Example}
\newtheorem{remark}[theorem]{Remark}

\theoremstyle{definition}

\newcommand{\isDistr}{\overset d=}

\newcommand{\dd}{\text{d}}

\newcommand{\PSRW}{P^{\textup{SRW}}}

\makeatletter
\def\namedlabel#1#2#3{\begingroup
    #2%
    \def\@currentlabel{#3}%
    \phantomsection\label{#1}\endgroup
}
\makeatother

\newcommand{\logp}{\log^+\hspace{-.1cm}}

\title{New characterization of the weak disorder phase of directed polymers in bounded random environments}
\author{Stefan Junk}
\address{University of Tsukuba}
\email{sjunk@math.tsukuba.ac.jp}
\date{\today}


\begin{document}

\lineskip=0pt

\begin{abstract}
We show that the weak disorder phase for the directed polymer model in a bounded random environment is characterized by the integrability of the running supremum $\sup_{n\in \N}W_n^\beta$ of the associated martingale $(W_n^\beta)_{n\in \N}$. Using this characterization, we prove that $(W_n^\beta)_{n\in \N}$ is $L^p$-bounded in the whole weak disorder phase, for some $p>1$. The argument generalizes to non-negative martingales with a certain product structure.
\end{abstract} 

\maketitle

\section{The directed polymer model}\label{sec:poly}

\subsection{Overview}\label{sec:overview}

\smallskip The \emph{directed polymer model} was introduced in the physics literature to describe the folding of long molecule chains in a solution with random impurities. Mathematically, it is a model for random paths, called \emph{polymers}, that are attracted or repulsed by a space-time random environment with a parameter $\beta\geq 0$, called \emph{inverse temperature}, governing the strength of the interaction. In recent years, the model has attracted much interest because it is conjectured that in a certain \emph{low temperature} regime it belongs to the KPZ (Kardar-Parisi-Zhang) universality class of randomly growing surfaces. 
In contrast, we focus on the \emph{high temperature} phase, where it is known that the influence of the disorder disappears asymptotically and that the long-term behavior is diffusive. This weak disorder phase is characterized by whether an associated martingale, $(W_n^\beta)_{n\in\N}$, is uniformly integrable, which is known to hold for small $\beta$ if the spatial dimension is at least three.

\smallskip There is no closed-form characterization for the critical inverse temperature $\beta_{cr}$ and, in practice, the uniform integrability is not easy to analyze. Several important features of the weak disorder phase have been established, see for example \cite{CY03,CY06,BC20}, but many more papers have focused on a different, \emph{very high temperature} phase, which is characterized by $L^2$-boundedness of the associated martingale. This condition is known to be strictly stronger than uniform integrability \cite{B04}, which naturally raises the question of whether the $L^2$-regime is relevant mainly because of its computational convenience, or 
whether the model undergoes a true phase transition within the weak disorder phase with some quantifiable change in behavior. 

\smallskip Theorem~\ref{thm:main} below is an indicatation that no such transition occurs. Namely, we show that $(W_n^\beta)_{n\in\N}$ is $L^p$-bounded in the whole weak disorder phase, for some $p>1$ depending on $\beta$. We believe that our result will be useful for extending results from the $L^2$-phase to the whole weak disorder phase. A first step in this direction will be presented in the follow-up paper \cite{J21_2}, see Section~\ref{sec:followup}. Theorem~\ref{thm:main} also suggests that the weak disorder phase might be better characterized by the integrability of $\sup_{n\in\N}W_n^\beta$, which is a strictly stronger condition than uniform integrability, see Section~\ref{sec:background}, although at present, we can only prove this for bounded environments.

\subsection{Definition}\label{sec:defs}
We introduce the directed polymer with \emph{site disorder}, i.e. with disorder attached to the sites of the directed graph $\N\times\Z^d$. 
Let $(\omega_0,\P_0)$ be a probability measure on $(\Omega_0,\F_0)=(\R,\mathcal B(\R))$, where we assume 
\begin{align}\label{eq:exp_mom}
\E\big[e^{\beta| \omega_0|}\big]<\infty\qquad\text{ for all }\beta\geq 0,
\end{align}
and let $\Omega=\Omega_0^{\N\times\Z^d}$ and $\big(\omega=(\omega_{t,x})_{t\in\N,x\in\Z^d},\P=\bigotimes_{(t,x)\in\N\times\Z^d}\P_0\big)$. The \emph{energy} of a path $x=(x_t)_{t\in\N}$ up to time $n$ in environment $\omega$  is
\begin{align*}
H_n(\omega,x)&\coloneqq \textstyle \sum_{t=1}^n\omega_{t,x_t}.
\end{align*}
Let $(X=(X_n)_{n\in\N},\PSRW)$ denote the simple random walk on $\Z^d$ and define the \emph{polymer measure} by
\begin{align*}
\mu_{\omega,n}^{\beta}(\dd X)&\coloneqq  (Z_n^{\beta}(\omega))^{-1}e^{\beta H_n(\omega,X)} \PSRW(\dd X),
\end{align*}
where $Z_n^{\beta}$ is the normalizing constant, called the \emph{partition function} of the model. Under $\mu_{\omega}^{\beta}$, paths are attracted to areas of space-time where the environment is positive and repelled by areas where it is negative. Note that this is not a consistent family of probability measures, i.e., there is no infinite volume probability measure ``$\mu_{\omega,\infty}^{\beta}$'' whose projection to time $n$ agrees with $\mu_{\omega,n}^\beta$, simultaneously for all $n\in\N$.  The \emph{associated martingale} mentioned in the previous section is defined by 
\begin{align*}
W_n^{\beta}(\omega)&\coloneqq  Z_n^{\beta}(\omega)e^{-n\lambda(\beta)},
\end{align*}
where $\lambda(\beta)\coloneqq \log \E[e^{\beta\omega_0}]$.
As a non-negative martingale, it is clear that the limit $W_\infty^{\beta}\geq 0$ exists almost surely. We say that \emph{weak disorder} \eqref{eq:WD}, resp. \emph{strong disorder} \eqref{eq:SD}, holds if
\begin{align}\label{eq:WD}\tag{WD}
\P(W_\infty^{\beta}>0)&>0.\\
\label{eq:SD}\tag{SD}
\P(W_\infty^\beta=0)&=1.
\end{align}
Using assumption \eqref{eq:exp_mom}, it is not hard to see that $W_\infty^\beta$ satisfies a zero-one law, i.e., for all $\beta\geq 0$,
\begin{align}\label{eq:zero_one}
\P(W_\infty^\beta>0)\in\{0,1\}.
\end{align}
Finally, we say that the environment is \emph{upper bounded} \eqref{eq:upper_bd} or \emph{lower bounded} \eqref{eq:lower_bd} if there exists $K>1$ such that, respectively,
\begin{align}
\P_0\left([K,\infty)\right)&=0,\tag{U-Bd.}\label{eq:upper_bd}\\
\P_0((-\infty,-K])&=0.\tag{L-Bd.}\label{eq:lower_bd}
\end{align}

\subsection{Known results}\label{sec:known}

We highlight a few key results and refer to \cite{C17} for a detailed survey of the model. 

\smallskip  It is known \cite[Theorem 1.1]{CY03} that there exists a critical temperature $\beta_{cr}=\beta_{cr}(d)$ such that weak disorder \eqref{eq:WD} holds for $\beta<\beta_{cr}$ and strong disorder \eqref{eq:SD} holds for $\beta>\beta_{cr}$, with $\beta_{cr}(d)>0$ if and only if $d\geq 3$. Moreover, \eqref{eq:WD} is equivalent to uniform integrability of $(W_n^\beta)_{n\in\N}$ \cite[Proposition 3.1]{CY06}. 

\smallskip As mentioned in Section~\ref{sec:overview}, many paper consider a stronger condition than uniform integrability, namely $L^2$-boundedness. The second moment of $W_n^\beta$ has a nice representation involving two independent random walks, called \emph{replicas} in this context, and in particular one can explicitly compute the critical temperature $\beta_{cr}^{L^2}$ for $L^2$-boundedness. In contrast, no closed form expression is known for $\beta_{cr}$, but \cite{B04} has shown that $\beta_{cr}^{L^2}<\beta_{cr}$. Several important results have been obtained for the whole weak disorder phase, most notably a central limit theorem in probability \cite[Theorem 1.2]{CY06} for $(\mu_{\omega,n}^{\beta})_{n\in\N}$, but our understanding of the $L^2$-phase is much more complete. 

\smallskip To give an example, it is known that convergence in probability in the central limit theorem can be replaced by almost sure convergence \cite{IS88,B89,SZ96} in the $L^2$-phase, and one would expect that this extends to the whole weak disorder phase. Based in part on the results of this paper, some progress towards this extension for related models is obtained in the follow-up paper \cite{J21_2}. Further results that are at present only known in $L^2$-weak disorder are \cite[Theorem 6.2]{CY06}, \cite{CN20} and \cite{S95,V06}. We believe that Theorem~\ref{thm:main} will be a useful tool towards extending them to the whole weak disorder phase.

\smallskip We also mention that much research has focused on the strong disorder phase \eqref{eq:SD}. In particular, the one-dimensional case is a very active field of research because of its conjectured relation to the KPZ universality class and because exactly solvable models are known. The behavior of $\mu_{\omega,n}^{\beta}$ in strong disorder is radically different from the weak disorder phase that is the focus of the present article.

\subsection{$L^p$-boundedness in weak disorder}

We state the main result of the paper.

\begin{theorem}\label{thm:main}
\begin{enumerate}[leftmargin=.7cm,label=(\roman*)]
 \item Assume \eqref{eq:upper_bd}. If \eqref{eq:WD} holds, then 
\begin{align}\label{eq:sup_W}
\E\left[\sup_{n\in\N}W_n^{\beta}\right]<\infty
\end{align}
and there exists $p>1$ such that
\begin{align}\label{eq:Lp_W}
\sup_{n\in \N}\|W_n^\beta\|_p<\infty.
\end{align}
The interval of $p>1$ satisfying \eqref{eq:Lp_W} is open. In contrast, if \eqref{eq:SD} holds then, for all $t>1$,
\begin{align}\label{eq:lower_W}
\P\left(\sup_{n\in\N}W_n^{\beta}>t\right)\geq \frac{1}{4K^2t}.
\end{align}
\item Assume that \eqref{eq:lower_bd} and \eqref{eq:WD} hold. There exists $\eps>0$ such that
\begin{align}\label{eq:Leps_W}
\sup_{n\in\N}\E\left[(W_n^\beta)^{-\eps}\right]<\infty.
\end{align}
The set of $\eps>0$ satisfying \eqref{eq:Leps_W} is open.
\end{enumerate}
\end{theorem}

In Section \ref{sec:gen_res}, we state a general result about non-negative martingales that implies Theorem \ref{thm:main}, and the proof is therefore postponed  until Section \ref{sec:proofs}. We make a few comments to put the result into perspective.

\begin{remark}
\begin{enumerate}[leftmargin=.7cm,label=(\roman*)]
 \item It is an interesting question whether \eqref{eq:upper_bd} is necessary. Theorem~\ref{thm:moments} below is certainly not valid without \eqref{eq:M_n_upper}, so a proof of \eqref{eq:Lp_W} without \eqref{eq:upper_bd} would have to be more specific to the the directed polymer model.
 \item The result can be taken as an indication that strong disorder \eqref{eq:SD} holds at $\beta_{cr}$. Indeed, otherwise Theorem~\ref{thm:main}(i) implies that $\sup_n\|W_n^{\beta_{cr}}\|_p<\infty$ for some $p>1$ while $(W_n^{\beta})_{n\in\N}$ is not uniformly integrable for any $\beta>\beta_{cr}$, which seems unlikely.
 \item In \cite{DE92}, a lower bound for $\beta_{cr}$ was obtained by estimating $\beta_{cr}^{L^p}$, the critical temperature for $L^p$-boundedness, for $p\in(1,2)$. Theorem~\ref{thm:main}(i) shows that such a lower bound can potentially be sharp. 
 \item One consequence of Theorem~\ref{thm:main}(i) is that in the $L^2$-phase, $\beta<\beta_{cr}^{L^2}$, there exists $\eps>0$ such that 
\begin{align*}
\sup_{n\in\N}\E\left[(W_n^\beta)^{2+\eps}\right]<\infty.
\end{align*}
The same result has been obtained with the help of hypercontractivity in \cite[display $(3.12)$]{FSZ20} in dimension $d=2$ in the so-called intermediate weak disorder phase.
\item We note that the zero-one law \eqref{eq:zero_one} is only necessary for part (ii). In particular, the conclusions of part (i) remain valid if assumption \eqref{eq:exp_mom} is replaced by $\E\left[e^{\beta\omega_0}\right]<\infty$  for all $\beta\geq 0$. 
Using the convension $e^{-\infty}=0$, we can therefore include degenerate models with $\P_0(\omega_0=-\infty)>0$. A natural example for such a degenerate model is oriented site or bond percolation, see \cite{Y08}. 
\end{enumerate}
\end{remark}

\subsection{Extensions and further research}\label{sec:followup}

We only present the result for the most commonly studied variation of this model, in discrete time and with the disorder attached to the sites of $\N\times\Z^d$, but Theorem~\ref{thm:moments} below applies to more generally:
\begin{itemize}
 \item The simple random walk $\PSRW$ can be replaced by other random walks.
 \item We can replace site disorder by \emph{bond disorder}.
 \item Finally, we may consider a model in continuous space-time with Poissonian disorder and $\PSRW$ replaced by Brownian motion.
\end{itemize}
All of these generalizations are discussed in the follow-up paper \cite{J21_2}. That paper also introduces a different tool for the directed polymer model, which is related to the so-called \emph{noise operator} appearing, for example, in the context of the hypercontractive inequality. For bond disorder, we give an alternative proof for the central limit theorem in probability from \cite{CY06} and we identify the large deviation rate function with that of the underlying random walk. For the Brownian polymer model in continuous space-time, we improve the convergence in the central limit theorem to almost sure convergence. To prove these results, we repeatedly use Theorem~\ref{thm:main}~(i) to justify interchanging limit and expectation.

\section{$L^p$-boundedness for martingales with product structure}\label{sec:general}

\subsection{Background}\label{sec:background}

Consider a non-negative martingale $(M_n)_{n\in\N}$ with $M_0=1$ and let $M_\infty:=\lim_{n\to\infty}M_n$ denote its almost sure limit. We have in mind a situation where $(M_n)_{n\in\N}$ is associated with a model undergoing a phase transition characterized by whether $M_\infty>0$ or $M_\infty=0$, similar to the directed polymer model. One can usually show that $M_\infty>0$ implies that $(M_n)_{n\in\N}$ is uniformly integrable, hence 
\begin{align}\label{eq:varphi}
\sup_n\E[\varphi(M_n)]<\infty
\end{align}
for some convex function $\varphi$ with $\lim_{x\to\infty}\frac{\varphi(x)}x=\infty$. In practice, \eqref{eq:varphi} is a rather weak integrability condition since $\varphi$ is not explicit and $\frac{\varphi(x)}x$ may grow extremely slow. We show that, under certain assumptions, $\P(M_\infty>0)>0$ implies 
\begin{enumerate}
 \item[(a)] $L^1$-boundedness of the running maximum $M_n^*:=\sup_{k=0,\dots,n}M_k$
 \item[(b)]  and $L^p$-boundedness of $(M_n)_{n\in\N}$ for some $p>1$. 
\end{enumerate}

Of course, (b) implies (a) by Doob's maximal inequality, but we need (a) to prove (b). We think it is surprising that in our fairly general setup, a \emph{lower} bound on $M_\infty$ yields an \emph{upper} bound on $M_n$. 

\smallskip We briefly give some context for both steps. For the implication (a), we note that $\E[M^*_\infty]<\infty$ is a stronger condition than uniform integrability, see \cite[Chapter II, Exercise 3.15]{RY99}. The structure of that counterexample is similar to the classical gambler's ruin problem, in that the martingale grows larger and larger before it is eventually absorbed in a much smaller value. To prove (a), we assume that $(M_n)_{n\in\N}$ has a certain ``product structure'' which excludes such behavior. Namely, under this assumption, the probability that $M_{k+l}$ is much smaller than $M_k$ can be compared with the probability that $M_{l}$ is much smaller than $M_0=1$, and hence with the probability that $M_\infty$ is small. The latter probability is bounded away from zero if $\P(M_\infty>0)>0$.

\smallskip For the implication (b), we recall from Doob's maximal inequality that
\begin{align*}
\|M_n^*\|_1\leq C(1+\E[M_n\logp M_n]).
\end{align*}
In \cite{G69}, it was shown that if $M_{n+1}/M_n$ is bounded, then the converse is also true, i.e. $\E[M_\infty^*]<\infty$ implies $\sup_n\E[M_n\logp M_n]<\infty$. 
Note that this is already an improvement on \eqref{eq:varphi}. The proof of the implication (b) is inspired by this result, in the sense that we use integrability of the running maximum $M_n^*$ to get a moment bound for $M_n$ itself. We obtain the stronger conclusion of $L^p$-boundedness by using the product structure introduced for implication (a).

\subsection{$L^p$-boundedness for certain martingales}\label{sec:gen_res}

We now state the main result of this section. The ``product structure'' mentioned above is \eqref{eq:assume_convex}.

\begin{theorem}\label{thm:moments}
Let $((\F_n)_{n\in\N},(M_n)_{n\in\N})$ be a non-negative martingale with $M_0=1$. Assume that for every $k,l\in\N$ and $f\colon\R_+\to\R$ convex, almost surely on $M_k>0$,
\begin{align}\label{eq:assume_convex}
\E\left[f\Big(\frac{M_{k+l}}{M_k}\Big)\Big|\F_k\right]\leq \E[f(M_l)].
\end{align}
Let $M_n^*:=\sup_{k=0,\dots,n}M_k$ and $M_\infty:=\lim_{n\to\infty}M_n$. Then the following hold.
\begin{enumerate}[label=(\roman*)]
 \item $\E[M_\infty^*]<\infty$ holds if 
\begin{align}\label{eq:M_infty_positive}
\P\left(M_\infty>0\right)>0.
\end{align}
 \item If \eqref{eq:M_infty_positive} holds and if there exists $K>1$ such that 
 \begin{align}
\P(M_{n+1}\leq KM_n)&=1\quad\text{ for all }n\in\N,\label{eq:M_n_upper}
\end{align}
then there exists $p>1$ such that
\begin{align}\label{eq:Lp_mart}
\sup_n\|M_n\|_p<\infty
\end{align}
Moreover, the interval of $p>1$ satisfying \eqref{eq:Lp_mart} is open.
 \item If $\P(M_\infty=0)=1$ and if \eqref{eq:M_n_upper} holds, then, for all $t>1$,
 \begin{align}\label{eq:strong}
\P(M_\infty^*>t)> \frac{1}{4K^2t}.
\end{align}

 \item If $\P(M_\infty>0)=1$ and if there exists $K>1$ such that 
 \begin{align}
\P(M_{n+1}\geq M_n/K)&=1\quad\text{ for all }n\in\N,\label{eq:M_n_lower}
\end{align}
then there exists $\eps>0$ such that
\begin{align}\label{eq:Leps}
\sup_n\E[M_n^{-\eps}]<\infty
\end{align} 
Moreover, the interval of $\eps>0$ satisfying \eqref{eq:Leps} is open.
\end{enumerate}

\end{theorem}

We give two examples to demonstrate that the assumption \eqref{eq:assume_convex} is natural, at least in context of branching processes. See \cite{AN72} and \cite{S15} for background on Galton-Watson processes and on branching random walks.


\begin{example}\label{ex:GWP}
Let $(Z_n)_{n\in\N}$ be a Galton-Watson process with $Z_0=1$ and associated martingale $M_n\coloneqq Z_n{m^{-n}}$, where $m\coloneqq \E[Z_1]$ is the expected number of offspring. Then, on $M_k>0$,
\begin{align*}
\E\left[f\Big(\frac{M_{k+l}}{M_k}\Big)\Big|\F_k\right]=
\E\left[f\Big(\frac{1}{Z_k}\sum_{i=1}^{Z_k} M_l^{(i)}\Big)\Big|\F_k\right]\leq \E\left[f\big(M_l^{(1)}\big)\Big|\F_k\right]=\E[f(M_l)],
\end{align*}
where 
$M_n^{(1)},M_n^{(2)},\dots$ are independent copies of 
$M_n$. Therefore \eqref{eq:assume_convex} holds.
\end{example}

\begin{example}\label{ex:branching}
We follow the notation from \cite{B79} for a one-dimensional branching random walk. For $\theta\in\R$, a martingale $(W^{(n)}(\theta))_{n\in\N}$ is defined by
\begin{align*}
\textstyle W^{(n)}(\theta)=(m(\theta))^{-n}\sum_{i=1}^{r^{(n)}}e^{-\theta z^{(n)}_i},
\end{align*}
where $r^{(n)}$ denotes the number of particles at time $n$ and $z^{(n)}_1,\dots,z^{(n)}_{r^{(n)}}$ their positions. To verify \eqref{eq:assume_convex}, let $\widetilde W^{(l,1)}(\theta),\widetilde W^{(l,2)}(\theta),\dots $ denote independent copies of $W^{(l)}(\theta)$. On $\{W^{(k)}(\theta)>0\}=\{r^{(k)}\geq 1\}$, we have
\begin{align}\label{eq:product}
\frac{W^{(k+l)}(\theta)}{W^{(k)}(\theta)}\isDistr\frac{\sum_{i=1}^{r^{(k)}}e^{-\theta z_i^{(k)}}\widetilde W^{(l,i)}(\theta)}{\sum_{i=1}^{r^{(k)}}e^{-\theta z^{(k)}_i}}=\sum_{i=1}^{r^{(k)}}\mu(\{i\})\widetilde W^{(l,i)}(\theta),
\end{align}
where $\mu$ is a probability measure on $\{1,\dots ,r^{(k)}\}$. Now \eqref{eq:assume_convex} follows as in Example~\ref{ex:GWP}.
\end{example}

The conclusion \eqref{eq:Lp_mart} is not new for Examples~\ref{ex:GWP} and \ref{ex:branching}, because \eqref{eq:M_n_upper} implies that the martingales converge in $L^p$ for some explicit $p>1$, see \cite[Part I.B, Theorem 2]{AN72} and \cite[p. 26]{B79}. Roughly speaking, the branching structure provides a lot of independence between particles in generation $n$, so the $p^{th}$ moment cannot be large if the influence from the early generations is small. This is not true for the directed polymer model, where the martingale has much larger correlations. The main observation in this paper is that a decomposition similar to \eqref{eq:product} is almost sufficient for $L^p$-boundedness.

\section{Proofs}\label{sec:proofs}

\begin{proof}[Proof of Theorem~\ref{thm:moments}]
\textbf{Part (i)}: We will find $\eps,\eta>0$ such that, for all $n\in\N$ and $t>1$,
\begin{align}\label{eq:todo}
\P(M_n^*>t)\leq \P(M_n>t\eps)/\eta,
\end{align}
so that
\begin{align*}
\E[M_n^*]\leq \frac{1}{\eps\eta}\E[M_n]+1=\frac{1}{\eps\eta}+1.
\end{align*}
Now, the right hand side does not depend on $n$ while the left hand side converges to $\E[M_\infty^*]$ for $n\to\infty$, so the claim follows. To prove \eqref{eq:todo}, we set, for some $\delta,\eps>0$ to be chosen later,
\begin{align*}
f_{\delta,\eps}(x)\coloneqq \delta  \Big(\frac {x}{\eps}-1\Big) \wedge 1.
\end{align*}
The important feature is that $f_{\delta,\eps}$ is concave and that, for all $x\geq 0$,
\begin{align}\label{eq:comp}
\ind_{[\eps,\infty)}(x)\geq f_{\delta,\eps}(x)\geq \ind_{[(\delta^{-1}+1)\eps,\infty)}(x)-\delta\ind_{[0,\eps]}(x).
\end{align}
Applying \eqref{eq:assume_convex} with $k=\tau$ and $l=n-\tau$, we get
\begin{equation}\label{eq:this} 
\begin{split}
\P(M_n>t\eps)&\geq \P\Big(\tau\leq n,\frac{M_n}{M_\tau}> \eps\Big)\\
&\geq \E\Big[\ind_{\tau\leq n}\E\Big[f_{\delta,\eps}\Big(\frac{M_n}{M_\tau}\Big)\Big|\F_{\tau}\Big]\Big]\\
&\geq \E\Big[\ind_{\tau\leq n}\E\Big[f_{\delta,\eps}\Big(\widetilde M_{n-\tau}\Big)\Big|\F_{\tau}\Big]\Big]\\
&\geq \P(\tau\leq n)\inf_k \E[f_{\delta,\eps}(M_{k})],
\end{split}
\end{equation}
where $\widetilde M_n$ is an independent copy of $M_n$. We have used \eqref{eq:comp} in the second inequality. For any $\delta>0$,
\begin{align*}
\inf_{k\in\N} \E[f_{\delta,\eps}(M_{k})]
&\geq \lim_{n\to\infty}\E\left[\inf_{k=0,...,n}f_{\delta,\eps}(M_{k})\right]\\
&\geq \E\left[\inf_{k\in\N}f_{\delta,\eps}(M_k)\right]\\
&\geq \P\Big(\inf_{k\in\N} M_k\geq (\delta^{-1} +1)\eps\Big)-\delta\P\left(\inf_{k\in\N} M_k\leq 
\eps\right)\\
&\xrightarrow{\eps\downarrow 0} \P(M_\infty>0)-\delta\P(M_\infty=0).
\end{align*}
The first inequality is Jensen's inequality, the second is Fatou's Lemma and the final inequality is \eqref{eq:comp}. In the last line, we used that $0$ is an absorbing state for non-negative martingales, hence
\begin{align}\label{eq:infimum}
\{M_\infty>0\}=\left\{\inf_{n\in\N}M_n>0\right\}.
\end{align}
Using \eqref{eq:M_infty_positive}, we find $\eps,\delta>0$ such that 
\begin{align*}
\inf_{k\in\N}\E[f_{\delta,\eps}(M_k)]=:\eta>0,
\end{align*}
which together with \eqref{eq:this} implies \eqref{eq:todo}.

\smallskip \textbf{Part (ii)}: We again use $\tau\coloneqq \inf\{n:M_n>t\}$, where $t>1$ will be chosen later. Assuption \eqref{eq:M_n_upper} implies that $M_n\leq tK\frac{M_n}{M_\tau}$ on $\{\tau\leq n\}$.
Therefore, for any $\eps>0$,
\begin{equation}\label{eq:nice}
\begin{split}
\E[M_n^{1+\eps}]&\leq t^{1+\eps}+\E[\ind_{\tau\leq n}M_n^{1+\eps}]\\
&\leq t^{1+\eps}+(Kt)^{1+\eps}\E\Big[\ind_{\tau\leq n}\Big(\frac{M_n}{M_\tau}\Big)^{1+\eps}\Big]\\
&\leq t^{1+\eps}+(Kt)^{1+\eps}\E\Big[\ind_{\tau\leq n}\E\left[(\widetilde M_{n-\tau})^{1+\eps}\Big|\F_\tau\right]\Big]\\
&\leq t^{1+\eps}+(Kt)^{1+\eps}\P(\tau\leq n)\E[M_n^{1+\eps}].
\end{split}
\end{equation}
The second-to-last inequality is \eqref{eq:assume_convex} and in the final inequality we used that $n\mapsto\E[f(M_n)]$ is increasing for any convex function $f$. By part (i),
\begin{align*}
\E[W^*_\infty]=\int_1^\infty \P(M_\infty^*>t)\dd t<\infty,
\end{align*}
so we find $t$ such that, for all $n\in\N$,
\begin{align*}
\P(\tau\leq n)\leq \P(W^*_\infty>t)\leq \frac{1}{4K^{2}t}.
\end{align*}
Once $t$ is fixed, we choose $\eps\in(0,1)$ such that $t^\eps\leq 2$, which by the above calculation implies
\begin{align*}
\E[M_n^{1+\eps}]\leq t^{1+\eps}+\frac 12\E[M_n^{1+\eps}].
\end{align*}
Since $n$ is arbitrary we get $\sup_n\E[M_n^{1+\eps}]\leq 2t^{1+\eps}$. Next, suppose that $\sup_n\|M_n\|_p<\infty$ for some $p>1$. Doob's maximal inequality implies that $\|M_\infty^*\|_p<\infty$ and, in particular, there exists $t>1$ such that 
\begin{align*}
\P(M_\infty^*>t)\leq \frac{1}{4K^{p+1}t^{p}}.
\end{align*}
Hence, using \eqref{eq:nice} with $q\in[p,p+1]$ in place of $1+\eps$, for any $n\in\N$,
\begin{align*}
\E[M_n^{q}]&\leq t^{q}+(Kt)^{q}\P(\tau\leq n)\E[M_n^{q}]\\
&\leq t^{q}+\frac{t^{q-p}}4\E[M_n^{q}].
\end{align*}
Now choose $q\in(p,p+1)$ such that $t^{q-p}\leq 2$, which shows $\sup_n\|M_n\|_q<\infty$.

\smallskip\textbf{Part (iii)}: If \eqref{eq:strong} fails for some $t>1$, the argument from part (ii) shows $\sup_n\E[M_n^{1+\eps}]\leq 2t^{1+\eps}<\infty$ for some $\eps>0$. Hence $(M_n)_{n\in\N}$ is uniformly integrable and $\E[M_\infty]=1$, contradicting $\P(M_\infty=0)=1$.

\smallskip\textbf{Part (iv)}: Using $\tau\coloneqq \inf\{n:M_n\leq 1/t\}$, similar calculations as in part (ii) give
\begin{align*}
\E[M_n^{-\eps}]\leq t^{\eps}+(Kt)^\eps \P(\tau\leq n)\E[M_n^{-\eps}].
\end{align*}
Using $\P(M_\infty>0)=1$ and \eqref{eq:infimum}, we have
\begin{align*}
\P(\tau\leq n)\leq \P\left(\inf_n M_n\leq 1/t\right)\xrightarrow{t\to\infty}0,
\end{align*}
so we can choose $t$ large enough that, for all $n\in\N$,
\begin{align*}
\P(\tau\leq n)\leq \frac{1}{4K^2}.
\end{align*}
Once $t$ is fixed, we choose $\eps\in(0,1)$ such that $t^\eps\leq 2$ and get
\begin{align*}
\sup_n\E[M_n^{-\eps}]<2t^\eps.
\end{align*}
\end{proof}


\begin{proof}[Proof of Theorem~\ref{thm:main}]
In view of Theorem~\ref{thm:moments}, we have to verify that $(W_n^\beta)_{n\in\N}$ satisfies \eqref{eq:assume_convex}. For $k,l\in\N$, let $Y$ be a $\Z^d$-valued random variable with law $\mu_{\omega,k}^\beta(X_k\in\cdot)$. Let
\begin{align*}
\widetilde W_l^\beta\coloneqq W_l^\beta\circ\theta_{k,Y},
\end{align*}
where $\theta_{k,y}$ denotes the space-time shift. Then $\widetilde W_l^\beta\isDistr W_l^\beta$, $\widetilde W_l^\beta$ is independent of $W_k^\beta$ and, by the Markov propery of $X$,
\begin{align}\label{eq:product2}
\frac{W_{k+l}^\beta}{W_k^\beta}
&=\sum_{y\in\Z^d}\mu_{\omega,k}^{\beta}(X_k=y)W_l^{\beta}\circ\theta_{k,y}=\E\big[\widetilde W_l^\beta\big|\omega\big].
\end{align}
Hence, by Jensen's inequality and the tower property of conditional expectation,
\begin{align*}
\E\left[f\left(\frac{W_{k+l}}{W_s}\right)\Big|\F_k\right]=\E\left[f\left(\E\big[\widetilde W_l^\beta\big|\omega\big]\right)\Big|\F_k\right]\leq \E\big[f\big(\widetilde W_l^\beta\big)\big|\F_k\big]=\E\big[f\big(W_l^\beta\big)\big].
\end{align*}
For part (ii), note that \eqref{eq:WD} implies $\P(W_\infty^\beta>0)=1$ by \eqref{eq:zero_one}, so Theorem~\ref{thm:moments}(iv) applies.
\end{proof}

\section*{Acknowledgements}
This work was supported by the JSPS Postdoctoral Fellowship for Research in Japan, Grant-in-Aid for JSPS Fellows 19F19814. The author is grateful to Ryoki Fukushima for many inspiring discussions and suggestions and to Shuta Nakajima for helpful suggestions.


\end{document}